\documentclass[draft]{amsart}
\usepackage{amssymb} 
\newtheorem{thm}{Theorem}[section]

\newtheorem{pro}[thm]{Proposition}
\newtheorem{DEF}[thm]{Definition}

\numberwithin{equation}{section}

\begin{document}
\title[]{SOME INTERSECTIONS OF LORENTZ SPACES}%
\author{ F. ABTAHI, H. G. Amini, H. A. Lotfi and A. Rejali}%

\thanks{}%


\date{}%


\maketitle

\begin{abstract}
Let $(X,\mu)$ be a measure space. For $p,q\in (0,\infty]$ and
arbitrary subsets $P,Q$ of $(0,\infty]$, we introduce and
characterize some intersections of Lorentz spaces, denoted by
$IL_{p,Q}(X,\mu)$, $IL_{J,q}(X,\mu)$ and $IL_{J,Q}(X,\mu)$.
\end{abstract}

\section{\bf Introduction}

Let $(X,\mu)$ be a measure space. For $0<p\leq\infty$, the space
$L^p(X,\mu)$ is the usual Lebesgue space as defined in \cite{gr}
and \cite{r}. Let us remark that for $1\leq p<\infty$
$$\|f\|_p:=\left(\int_X|f(x)|^pd\mu(x)\right)^{1/p}$$ defines a
norm on $L^p(X,\mu)$ such that $(L^p(X,\mu),\|.\|_p)$ is a Banach
space. Also for $0<p<1$,
$$
\|f\|_p:=\int_X|f(x)|^pd\mu(x)
$$
defines a quasi norm on $L^p(X,\mu)$ such that
$(L^p(X,\mu),\|.\|_p)$ is a complete metric space. Moreover for
$p=\infty$,
$$
\|f\|_{\infty}=\inf\{B>0: \mu(\{x\in X: |f(x)|>B\})=0\}
$$
defines a norm on $L^{\infty}(X,\mu)$ such that
$(L^{\infty}(X,\mu),\|.\|_{\infty})$ is a Banach space. In
\cite{aalr1}, we considered an arbitrary intersection of the
$L^p-$spaces denoted by $\bigcap_{p\in J}L^p(G)$, where $G$ is a
locally compact group with a left Haar measure $\lambda$ and
$J\subseteq [1,\infty]$. Then we introduced the subspace $IL_J(G)$
of $\bigcap_{p\in J}L^p(G)$ as
$$
IL_{J}(G)=\{f\in\bigcap_{p\in J} L^p(G): \|f\|_J=\sup_{p\in
J}\Vert f\Vert_{p}<\infty\},
$$
and studied $IL_J(G)$ as a Banach algebra under convolution
product, for the case where $1\in J$. Also in \cite{aalr2}, we
generalized the results of \cite{aalr1} to the weighted case. In
fact for an arbitrary family $\Omega$ of the weight functions on
$G$ and $1\leq p<\infty$, we introduced the subspace
$IL_p(G,\Omega)$ of the locally convex space
$L^p(G,\Omega)=\bigcap_{\omega\in\Omega}L^p(G,\omega)$. Moreover,
we provided some sufficient conditions on $G$ and also $\Omega$ to
construct a norm on $IL_p(G,\Omega)$. The fourth section of
\cite{aalr2} has been assigned to some intersections of Lorentz
spaces. Indeed for the case where $p$ is fixed and $q$ runs
through $J\subseteq (0,\infty)$, we introduced $IL_{p,J}(G)$ as a
subspace of $\cap_{q\in J}L_{p,q}(G)$, where $L_{p,q}(G)$ is the
Lorentz space with indices $p$ and $q$. As the main result, we
proved that $IL_{p,J}(G)=L_{p,m_J}(G)$, in the case where
$m_J=\inf\{q: q\in J\}$ is positive.\\

In the present work, we continue our study concerning the
intersections of Lorentz spaces on the measure space $(X,\mu)$, to
complete our results in this direction. Precisely, we verify most
the results given in the second and third sections of
\cite{aalr1}, for Lorentz spaces.

\section{\bf Preliminaries}

In this section, we give some preliminaries and definitions which
will be used throughout the paper. We refer to \cite{gr}, as a
good
introductory book.\\

Let $(X,\mu)$ be a measure space and $f$ be a complex valued
measurable function on $X$. For each $\alpha>0$, let
$$
d_f(\alpha)=\mu(\{x\in X: |f(x)|>\alpha \}).
$$
The decreasing rearrangement of $f$ is the function
$f^*:[0,\infty)\rightarrow [0,\infty]$ defined by
$$
f^*(t)=\inf\{s>0: d_f(s)\leq t\}.
$$
We adopt the convention $\inf\emptyset=\infty$, thus having
$f^*(t)=\infty$ whenever $d_f(\alpha)>t$ for all $\alpha\geq 0$.
For $0<p\leq\infty$ and $0<q<\infty$, define
\begin{equation}\label{e1}
\|f\|_{L_{p,q}}
=\left(\int_0^{\infty}\left(t^{\frac{1}{p}}f^*(t)\right)^q\frac{dt}{t}\right)^{1/q},
\end{equation}
where $dt$ is the Lebesgue measure. In the case where $q=\infty$,
define
\begin{equation}\label{e3}
\|f\|_{L_{p,\infty}}=\sup_{t>0} t^{\frac{1}{p}}f^*(t).
\end{equation}
The set of all $f$ with $\|f\|_{p,q}<\infty$ is denoted by
$L_{p,q}(X,\mu)$ and is called the Lorentz space with indices $p$
and $q$. As in $L^p-$spaces, two functions in $L_{p,q}(X,\mu)$ are
considered equal if they are equal $\mu-$almost everywhere on $X$.
It is worth noting that by \cite[Proposition 1.4.5]{gr} for each
$0<p<\infty$ we have
\begin{equation}\label{e2}
\int_X|f(x)|^pd\mu(x)=\int_0^{\infty}f^*(t)^pdt.
\end{equation}
It follows that $L_{p,p}(X,\mu)=L^p(X,\mu)$. Furthermore by the
definition given in equation (\ref{e3}), one can observe that
$L_{\infty,\infty}(X,\mu)=L^{\infty}(X,\mu)$. Note that in the
case where $p=\infty$, one can conclude that the only simple
function with finite $\|.\|_{L_{\infty,q}}$ norm is the zero
function. For this reason, $L_{\infty,q}(X,\mu)=\{0\}$, for every
$0<q<\infty$; see \cite[page 49]{gr}.\\

In \cite{aalr2}, for locally compact group $G$ and $0<p<\infty$
and also an arbitrary subset $Q$ of $(0,\infty)$ with
$$m_Q=\inf\{q: q\in Q\}>0,$$ we introduced $IL_{p,Q}(G)$ as a subset
of $\cap_{q\in Q}L_{p,q}(G)$ by
\begin{equation}\label{e4}
IL_{p,Q}(G)=\{f\in \bigcap_{q\in Q}L_{p,q}(G):
\|f\|_{L_{p,Q}}=\sup_{q\in Q}\|f\|_{L_{p,q}}<\infty\}.
\end{equation}
As the main result of the third section in \cite{aalr2}, we proved
the following theorem.

\begin{thm}\cite[Theorem 12]{aalr2}
Let $G$ be a locally compact group, $0<p<\infty$ and $Q$ be an
arbitrary subset of $(0,\infty)$ such that $m_Q>0$. Then
$IL_{p,Q}(G)=L_{p,m_Q}(G)$. Moreover, for each $f\in
L_{p,m_Q}(G)$,
\begin{equation}\label{e8}
\|f\|_{L_{p,m_Q}}\leq\|f\|_{L_{p,Q}}\leq\max\left\{1,
\left(\frac{m_Q}{p}\right)^{1/{m_Q}}\right\}\|f\|_{L_{p,m_Q}}.
\end{equation}
\end{thm}
Note that in the definition of $IL_{p,Q}(G)$ given in (\ref{e4}),
one can replace $G$ by an arbitrary measure space $(X,\mu)$.
Precisely if let
\begin{equation}\label{e7}
IL_{p,Q}(X,\mu)=\{f\in \bigcap_{q\in Q}L_{p,q}(X,\mu):
\|f\|_{L_{p,Q}}=\sup_{q\in Q}\|f\|_{L_{p,q}}<\infty\},
\end{equation}
then $IL_{p,Q}(X,\mu)=L_{p,m_Q}(X,\mu)$. Moreover for each $f\in
L_{p,m_Q}(X,\mu)$, \eqref{e8} is satisfied. Furthermore, Theorem
\cite[Theorem 12]{aalr2} is also valid for $IL_{p,Q}(X,\mu)$. In
the present work, in a similar way, we introduce and characterize
the spaces $IL_{J,q}(X,\mu)$ and also $IL_{J,Q}(X,\mu)$, as other
intersections of Lorentz spaces. Moreover we obtain some results
about Lorentz space related to a Banach spaces $E$, introduced in
\cite{kat}.

\section{\bf Main Results}

At the beginning of the present section we recall \cite[Exercise
1.4.2]{gr}, which will be used several times in our further
arguments. A simple proof is given here.

\begin{pro}\label{p1}
Let $(X,\mu)$ be a measure space and $0<p_1<p_2 \leq \infty$. Then
$$
L_{p_1,\infty}(X,\mu)\bigcap
L_{p_2,\infty}(X,\mu)\subseteq\bigcap_{p_1<p<p_2 ,
0<s\leq\infty}L_{p,s}(X,\mu).
$$
\end{pro}

\begin{proof}
Let $f\in L_{p_1,\infty}(X,\mu)\cap L_{p_2,\infty}(X,\mu)$. If
$\|f\|_{L_{p_1,\infty}}=0$, one can readily obtained that $f\in
L_{p,s}(X,\mu)$, for all $p_1<p<p_2$ and $0<s\leq\infty$. Now let
$\|f\|_{L_{p_1,\infty}}\neq 0$ and first suppose that
$p_2<\infty$. We show that $f\in L_{p,s}(X,\mu)$, for all
$p_1<p<p_2$ and $0<s<\infty$. It is clear that for each $\alpha>0$
\begin{equation}\label{e6}
d_f(\alpha)\leq \min
\left(\frac{\|f\|_{L_{p_1,\infty}}^{p_1}}{\alpha^{p_1}} ,
\frac{\|f\|_{L_{p_2,\infty}}^{p_2}}{\alpha^{p_2}}\right).
\end{equation}
Set
$$
B=\left(
\frac{\|f\|_{L_{p_2,\infty}}^{p_2}}{\|f\|_{L_{p_1,\infty}}^{p_1}}\right)^{\frac{1}{p_2-p_1}}.
$$
Thus
\begin{eqnarray*}
\|f\|_{L_{p,s}}&=&\left(p\int_0^{\infty}(d_{f}(\alpha)^{\frac{1}{p}}\;\alpha)^s
\frac{d\alpha}{\alpha}\right)^{\frac{1}{s}}=
\left(p\int_0^\infty d_f(\alpha)^{\frac{s}{p}}\;\alpha^{s-1}d\alpha\right)^{\frac{1}{s}}\\
&\leq& \left(p\int_0^B
\alpha^{s-1}\;\left(\frac{\|f\|_{L_{p_1,\infty}}^{p_1}}{\alpha^{p_1}}\right)^
{\frac{s}{p}}d\alpha\right)^{\frac{1}{s}}+ \left(p\int_B^{\infty}
\alpha^{s-1}\;\left(\frac{\|f\|_{L_{p_2,\infty}}
^{p_2}}{\alpha^{p_2}}\right)^{\frac{s}{p}}d\alpha\right)^{\frac{1}{s}}\\
&=& p^{\frac{1}{s}}\;
\|f\|_{L_{p_1,\infty}}^{\frac{p_1}{p}}\;\left(\int_0^B
\alpha^{s-1-\frac{sp_1}{p}}d\alpha\right)^{\frac{1}{s}}+
p^{\frac{1}{s}}\;\|f\|_{L_{p_2,\infty}}^{\frac{p_2}{p}}\;
\left(\int_B^{\infty} \alpha^{s-1-\frac{sp_2}{p}}d\alpha\right)^{\frac{1}{s}}\\
&=& p^{\frac{1}{s}}\;\|f\|_{L_{p_1,\infty}}^{\frac{p_1}{p}}\left
(\frac{B^{s-\frac{sp_1}{p}}}{s-\frac{sp_1}{p}}\right)^{\frac{1}{s}}
+ p^{\frac{1}{s}}\;\|f\|_{L_{p_2,\infty}}^{\frac{p_2}{p}}\left
(\frac{B^{s-\frac{sp_2}{p}}}{\frac{sp_2}{p}-s}\right)^{\frac{1}{s}}\\
&=& \left(\frac{p}{(s-\frac{sp_1}{p})}\right)^{\frac{1}{s}}\;
\|f\|_{L_{p_1,\infty}}^{\frac{p_1}{p}.(\frac{p_2-p}{p_2-p_1})}\;
\|f\|_{L_{p_2,\infty}}^{\frac{p_2}{p}\;(\frac{p-p_1}{p_2-p_1})}\\
&&+\left(\frac{p}{(\frac{sp_2}{p}-s)}\right)^{\frac{1}{s}}\;
\|f\|_{L_{p_1,\infty}}^{\frac{p_1}{p}.(\frac{p_2-p}{p_2-p_1})}\;
\|f\|_{L_{p_2,\infty}}^{\frac{p_2}{p}\;(\frac{p-p_1}{p_2-p_1})} \\
&=& \left(\left(\frac{p}{(s-\frac{sp_1}{p})}\right)^{\frac{1}{s}}+
\left(\frac{p}{(\frac{sp_2}{p}-s)}\right)^{\frac{1}{s}}\right)
\;\|f\|_{L_{p_1,\infty}}^{\frac{p_1}{p}(\frac{p_2-p}{p_2-p_1})}\;
\|f\|_{L_{p_2,\infty}}^{\frac{p_2}{p}(\frac{p-p_1}{p_2-p_1})}\\
&<&\infty.
\end{eqnarray*}
Thus $f\in L_{p,s}(X,\mu)$. For $p_2=\infty$, since
$d_f(\alpha)=0$ for each $\alpha
>\|f\|_{\infty}$, inequality (\ref{e6}) implies that
$$\|f\|_{L_{p,s}}^s \leq \frac{p}{s-\frac{sp_1}{p}}
\|f\|_{p_1,\infty}^{\frac{sp_1}{p}}\|f\|_{\infty}^{s-\frac{sp_1}{p}},$$
which implies $f\in L_{p,s}(X,\mu)$. In the case where $s=\infty$,
by \cite[Proposition 1.1.14]{gr}, for $p_1 < r <p_2$ we have
$$L_{p_1,\infty}(X,\mu)\cap L_{p_2,\infty}(X,\mu)\subseteq L^r(X,\mu)\subseteq L_{r,\infty}(X,\mu).$$
This gives the proposition.
\end{proof}

\begin{pro}\label{p5}
Let $(X,\mu)$ be a measure space, $0<q\leq\infty$ and
$0<p_1<p_2\leq \infty$. Then
$$
\bigcap_{p_1\leq r\leq p_2} L_{r,q}(X,\mu)=L_{p_1,q}(X,\mu)\cap
L_{p_2,q}(X,\mu).
$$
Moreover for each $f\in L_{p_1,q}(X,\mu)\cap L_{p_2,q}(X,\mu)$ and
$p_1<r<p_2$,
$$
\|f\|_{L_{r,q}}\leq 2^{1/q}\max\{\|f\|_{L_{p_1,q}},
\|f\|_{L_{p_2,q}}\}.
$$
\end{pro}

\begin{proof}
By \cite[Proposition 1.4.10]{gr} and  Proposition \ref{p1} we have
\begin{eqnarray*}
L_{p_1,q}(X,\mu)\cap
L_{p_2,q}(X,\mu)&\subseteq&L_{p_1,\infty}(X,\mu)\cap
L_{p_2,\infty}(X,\mu)\\
&\subseteq&\bigcap_{p_1<r<p_2 , 0<s\leq\infty} L_{r,s}(X,\mu).
\end{eqnarray*}
It follows that $$L_{p_1,q}(X,\mu)\cap L_{p_2,q}(X,\mu)\subseteq
\bigcap_{p_1\leq r\leq p_2} L_{r,q}(X,\mu).$$ The converse of the
inclusion is clearly valid. Thus
$$L_{p_1,q}(X,\mu)\cap L_{p_2,q}(X,\mu)=\bigcap_{p_1\leq
r\leq p_2} L_{r,q}(X,\mu).
$$
Now let $q<\infty$. For each $f\in L_{p_1,q}(X,\mu)\cap
L_{p_2,q}(X,\mu)$, we have
\begin{eqnarray*}
\|f\|_{L_{r,q}}^q &=& \int_0^{\infty}\left(t^{\frac{1}{r}}f^*(t)\right)^q\frac{dt}{t}\\
&\leq&\int_0^1\left(t^{\frac{1}{p_2}}f^*(t)\right)^q\frac{dt}{t}+\int_1^{\infty}
\left(t^{\frac{1}{p_1}}f^*(t)\right)^q\frac{dt}{t}\\
&\leq&\|f\|_{L_{p_2,q}}^q+ \|f\|_{L_{p_1,q}}^q\\
&\leq& 2\;\max\{\|f\|_{L_{p_2,q}}^q,\|f\|_{L_{p_1,q}}^q\}
\end{eqnarray*}
Also for $q=\infty$ we have
\begin{eqnarray*}
\|f\|_{L_{r,\infty}} &=& \sup_{t>0} t^{\frac{1}{r}}f^*(t) \\
&\leq& \max \{\sup_{0<t<1} t^{\frac{1}{p_2}}f^*(t),  \sup_{t\geq 1}t^{\frac{1}{p_1}}f^*(t)\} \\
&\leq& \max \{\|f\|_{L_{p_2,\infty}} , \|f\|_{L_{p_1,\infty}}\}.
\end{eqnarray*}
This completes the proof.
\end{proof}

We are in a position to prove \cite[Proposition 2.3]{aalr1} for
Lorentz spaces. It is obtained in the following proposition.
Recall from \cite{aalr1} that for a subset $J$ of $(0,\infty)$,
$$M_J=\sup\{p: p\in J\}.$$
\begin{pro}\label{p6}
Let $(X,\mu)$ be a measure space, $0<q\leq \infty$ and $J$ be a
subset of $(0,\infty)$ such that $0<m_J$. Then the following
assertions hold.
\begin{enumerate}
\item[(i)] If $m_J, M_J\in J$, then
$$\bigcap_{p\in[m_J,M_J]}L_{p,q}(X,\mu)=\bigcap_{p\in
J}L_{p,q}(X,\mu)=L_{m_J,q}(X,\mu)\cap L_{M_J,q}(X,\mu).$$
\item[(ii)] If $m_J\in J$ and $M_J\notin J$, then $\bigcap_{p\in
J}L_{p,q}(X,\mu)=\bigcap_{p\in [m_J,M_J)}L_{p,q}(X,\mu)$.
\item[(iii)] If $m_J\notin J$ and $M_J\in J$, then $\bigcap_{p\in
J}L_{p,q}(X,\mu)=\bigcap_{p\in (m_J,M_J]}L_{p,q}(X,\mu)$.
\item[(iv)] If $m_J,M_J\notin J$, then $\bigcap_{p\in
J}L_{p,q}(X,\mu)=\bigcap_{p\in (m_J,M_J)}L_{p,q}(X,\mu)$.
\end{enumerate}
\end{pro}

\begin{proof}
$(i)$. It is clearly obtain by Proposition \ref{p5}.

$(ii)$. Let $f\in \bigcap_{p\in J}L_{p,q}(X,\mu)$ and take $m_J <
t<M_J$. Then there exist $t_1,t_2\in J$ such that $t_1<t<t_2$. So
by Proposition \ref{p5}
$$
f\in L_{t_1,q}(X,\mu)\cap L_{t_2,q}(X,\mu)=\bigcap_{t_1\leq p\leq
t_2} L_{p,q}(X,\mu)
$$
and thus $f\in L_{t,q}(X,\mu)$. It follows that
$$
\bigcap_{p\in J}L_{p,q}(X,\mu)\subseteq L_{t,q}(X,\mu),
$$
for each $t\in [m_J,M_J)$. Consequently
$$
\bigcap_{p\in J}L_{p,q}(X,\mu)\subseteq\bigcap_{p\in
[m_J,M_J)}L_{p,q}(X,\mu).
$$
The converse of the inclusion is clear.

$(iii)$ and $(iv)$ are proved in a similar way.
\end{proof}

\noindent Similar to the definition of $IL_{p,Q}(X,\mu)$ given in
(\ref{e7}), for $J, Q\subseteq (0,\infty)$ let
\begin{center}
$IL_{J,q}(X,\mu)=\{f\in\bigcap_{p\in J} L_{p,q}(X,\mu):\;\;
\|f\|_{L_{J,q}}=\sup_{p\in J} \|f\|_{L_{p,q}}<\infty\}$
\end{center}
and
\begin{center}
$IL_{J,Q}(X,\mu)=\{f\in\bigcap_{p\in J,q\in Q} L_{p,q}(X,\mu):\;\;
\|f\|_{L_{J,Q}}=\sup_{p\in J,q\in Q} \|f\|_{L_{p,q}}<\infty\}$.
\end{center}

\begin{pro}\label{p7}
Let $(X,\mu)$ be a measure space, $0<q\leq\infty$ and $J\subseteq
(0,\infty)$ such that $m_J>0$. Then
$$IL_{J,q}(X,\mu)\subseteq
L_{m_J,q}(X,\mu)\cap L_{M_J,q}(X,\mu).$$
\end{pro}

\begin{proof}
First, let $q<\infty$. We follow a proof similar to the proof of
\cite[Theorem 12]{aalr2}. Suppose that $M_J<\infty$ and $(x_n)$ is
a sequences in $J$ such that $\lim_n x_n=M_J$. For $f\in
IL_{J,q}(X,\mu)$ by Fatou's lemma, we have
\begin{eqnarray*}
\|f\|_{L_{M_J,q}}^q &=&\int_0^{\infty}
\left(t^{\frac{1}{M_J}}.f^*(t)\right)^q\frac{dt}{t}\\
&=&\int_0^{\infty} \liminf_n\left(t^{\frac{1}{x_n}}.f^*(t)\right)^q \frac{dt}{t}\\
&\leq& \liminf_n \int_0^{\infty}\left(t^{\frac{1}{x_n}}.f^*(t)\right)^q \frac{dt}{t}\\
&=&\liminf_n \|f\|_{L_{x_n,q}}^q\\
&\leq&\|f\|_{J,q}^q\\
&<&\infty.
\end{eqnarray*}
If $M_J=\infty$ and $f\in IL_{J,q}(X,\mu)$, then
\begin{eqnarray*}
\left(\int_0^{\infty}f^*(t)^q\frac{dt}{t}\right)^{1/q}&=&\left(\int_0^{\infty}
\liminf_n
\left(t^{\frac{1}{x_n}}f^*(t)\right)^q\frac{dt}{t}\right)^{1/q}\\
&\leq&\liminf_n \|f\|_{L^{x_n,q}}\leq \|f\|_{J,q}<\infty.
\end{eqnarray*}
On the other hand, as we mentioned in section 1, since $q<\infty$
then $L_{\infty,q}(X,\mu)=\{0\}$ and since
$\int_0^{\infty}f^*(t)^q\frac{dt}{t}<\infty$, so we have $f=0$,
$\mu-$almost every where on $X$. Thus
$IL_{J,q}(X,\mu)=L_{\infty,q}(X,\mu)=\{0\}$. It follows that
$IL_{J,q}(X,\mu)\subseteq L_{M_J,q}(X,\mu)$.

Now suppose that $q=\infty$ and $f\in IL_{J,q}(X,\mu)$. Then
\begin{eqnarray*}
\|f\|_{L_{M_J,\infty}}&=&\sup_{t>0}t^{\frac{1}{M_J}}\;f^*(t)=\sup_{t>0}\left(\lim_n t^{\frac{1}{x_n}}.f^*(t)\right)\\
&\leq& \sup_{t>0}\left(\lim_n \|f\|_{L_{x_n,\infty}}\right)
=\|f\|_{L_{J,\infty}}< \infty,
\end{eqnarray*}
and so $f\in L_{M_J,\infty}(X,\mu)$. Thus we proved that
$IL_{J,q}(X,\mu)\subseteq L_{M_J,q}(X,\mu)$, for each
$0<q\leq\infty$. Using some similar arguments, one can obtain that
$IL_{J,q}(X,\mu)\subseteq L_{m_J,q}(X,\mu)$. Consequently the
proof is complete.
\end{proof}

The following proposition is obtained immediately from
Propositions \ref{p5}, \ref{p6} and \ref{p7}.

\begin{pro}\label{p8}
Let $(X,\mu)$ be a measure space, $0<q\leq \infty$ and $J\subseteq
(0,\infty)$ such that $m_J>0$ and $M_J<\infty$. Then
\begin{eqnarray*}
IL_{J,q}(X,\mu)&=& IL_{(m_J,M_J),q}(X,\mu)=IL_{[m_J,M_J),q}(X,\mu)\\
&=& IL_{(m_J,M_J],q}(X,\mu)=IL_{[m_J,M_J],q}(X,\mu)\\
&=&L_{m_J,q}(X,\mu)\cap L_{M_J,q}(X,\mu).
\end{eqnarray*}
Furthermore, for each $f\in IL_{J,q}(X,\mu)$ and $p\in J$,
$$
\|f\|_{L_{p,q}}\leq 2^{1/q}
\max\{\|f\|_{L_{m_J,q}},\|f\|_{L_{M_J,q}}\}
$$
\end{pro}

\begin{thm}\label{p9}
Let $(X,\mu)$ be a measure space and $J,Q\subseteq (0,\infty)$
such that $m_J,m_Q>0$. Then
\begin{equation}\label{e9}
IL_{J,Q}(X,\mu)=L_{m_J,m_Q}(X,\mu)\cap L_{M_J,m_Q}(X,\mu).
\end{equation}
Moreover for each $f\in IL_{J,Q}(X,\mu)$
\begin{eqnarray*}
\max\{\|f\|_{L_{m_J,m_Q}}, \|f\|_{L_{M_J,m_Q}}\}&\leq&\sup_{p\in
J,q\in Q} \|f\|_{L_{p,q}}\\
&\leq& K\max \{\|f\|_{L_{m_J,m_Q}}, \|f\|_{L_{M_J,m_Q}}\},
\end{eqnarray*}
for some positive constant $K>0$.
\end{thm}

\begin{proof}
Let $f\in IL_{J,Q}(X,\mu)$. Then by proposition \ref{p8}, we have
$$f\in L_{m_J,q}(X,\mu)\cap L_{M_J,q}(X,\mu),$$
for each $q\in Q$, and so
$$
f\in\left(\cap_{q\in Q} L_{m_J,q}(X,\mu)\right)\cap
\left(\cap_{q\in Q}L_{M_J,q}(X,\mu)\right).
$$
Thus \cite[Theorem 12]{aalr2} implies that $f\in
L_{m_J,m_Q}(X,\mu)\cap L_{M_J,m_Q}(X,\mu)$. Also by Fatou's lemma,
one can readily obtain that
$$
\|f\|_{L_{m_J,m_Q}}\leq\sup_{p\in J,q\in Q} \|f\|_{L_{p,q}}<\infty
$$
and also
$$
\|f\|_{L_{M_J,m_Q}}\leq\sup_{p\in J,q\in Q}
\|f\|_{L_{p,q}}<\infty.
$$
For the converse, note that by Proposition \ref{p5} and
\cite[Proposition 1.4.10]{gr}, we have
\begin{eqnarray*}
L_{m_J,m_Q}(X,\mu)\cap L_{M_J,m_Q}(X,\mu)&=&\bigcap_{m_J\leq r\leq
M_J} L_{r,m_Q}(X,\mu)\\
&\subseteq&\bigcap_{m_J\leq r\leq M_J, m_Q\leq t\leq
M_Q}L_{r,t}(X,\mu).
\end{eqnarray*}
It follows that $$L_{m_J,m_Q}(X,\mu)\cap
L_{M_J,m_Q}(X,\mu)\subseteq\bigcap_{p\in J,q\in Q}
L_{p,q}(X,\mu).$$ By Proposition \ref{p8} and \cite[Theorem
12]{aalr2}, for each $f\in L_{m_J,m_Q}(X,\mu)\cap
L_{M_J,m_Q}(X,\mu)$ we have
\begin{eqnarray*}
\max\{\|f\|_{L_{m_J,m_Q}} , \|f\|_{L_{M_J,m_Q}}\}
&\leq&\sup_{m_J\leq p\leq M_J,m_Q\leq q\leq M_Q}\|f\|_{L_{p,q}}\\
&\leq&\sup_{m_J\leq p\leq M_J}\left[\max\{1,(\frac{m_Q}{p})^{\frac{1}{m_Q}}\}\;\|f\|_{L_{p,m_Q}}\right]\\
&\leq&2^{1/{m_Q}} \max\{1,(\frac{m_Q}{m_J})^{\frac{1}{m_Q}}\}\;
\max\{\|f\|_{L_{m_J,m_Q}},\|f\|_{L_{M_J,m_Q}}\}
\end{eqnarray*}
and so the inequality is provided by choosing
$$K=2^{1/{m_Q}}\max\{1,(\frac{m_Q}{m_J})^{\frac{1}{m_Q}}\}.$$
Moreover $f\in IL_{J,Q}(X,\mu)$ and the equality \eqref{e9} is
satisfied.
\end{proof}

\begin{pro}\label{p12}
Let $(X,\mu)$ be a measure space and $0<p\leq\infty$. Then $fg\in
L_{p,\infty}(X,\mu)$, for each $f\in L^{\infty}(X,\mu)$ and $g\in
L_{p,\infty}(X,\mu)$.
\end{pro}

\begin{proof}
By \cite[Proposition 1.4.5]{gr} parts $(7)$ and $(15)$, we have
\begin{eqnarray*}
\|f\;g\|_{p,\infty}&=&\sup_{t>0}\left(t^{\frac{1}{p}}\;(f\;g)^*(t)\right)
\leq\sup_{t>0}\left(t^{\frac{1}{p}} f^*(\frac{t}{2})\;g^*(\frac{t}{2})\right) \\
&=& \sup_{t>0}\left((2t)^{\frac{1}{p}}\;f^*(t)\;g^*(t)\right)
\leq 2^{\frac{1}{p}}\;\|g\|_{p,\infty}\;\|f\|_{\infty}\\
&<&\infty.
\end{eqnarray*}
It follows that $fg\in L_{p,\infty}(X,\mu)$.
\end{proof}

\begin{pro}\label{p13}
Let $(X,\mu)$ be a measure space, $0<p\leq \infty$ and
$J,Q\subseteq (0,\infty)$ such that $m_J>0$, $m_Q>0$ and $m_Q\in
Q$. Then $IL_{J,Q}(X,\mu)=A\cap B$, where
$$
A=\{f\in\bigcap_{p\in J,q\in Q} L_{p,q}(X,\mu),\; M_q= \sup_{p\in
J} \|f\|_{L_{p,q}}<\infty,~ \forall q\in Q\}
$$
and
$$
B=\{f\in \bigcap_{p\in J,q\in Q} L_{p,q}(X,\mu),\; M_p= \sup_{q\in
Q} \|f\|_{L_{p,q}}<\infty,~ \forall p\in J\}.
$$
\end{pro}

\begin{proof}
It is clear that $IL_{J,Q}(X,\mu)\subseteq A\cap B$. For the
converse assume that $f\in A\cap B$. By \cite[Theorem 12]{aalr2}
implies that for each $p\in J$
$$
\sup_{q\in Q} \|f\|_{L_{p,q}} \leq \max\{1,
(\frac{m_Q}{p})^{\frac{1}{m_Q}}\} \|f\|_{L_{p,m_Q}}
$$
and so
\begin{eqnarray*}
\sup_{p\in J,~q\in Q} \|f\|_{L_{p,q}} &\leq&\max\{1,(\frac{m_Q}{m_J})^{\frac{1}{m_Q}}\}\sup_{p\in J}\|f\|_{L_{p,m_Q}}\\
&=&\max\{1,(\frac{m_Q}{m_J})^{\frac{1}{m_Q}}\}\;M_{m_Q}\\
&<&\infty.
\end{eqnarray*}
It follows that $f\in IL_{J,Q}(X,\mu)$.
\end{proof}

In the sequel, we investigate some previous results, for the
special Lorentz space $\ell_{p,q}\{E\}$, introduced in \cite{kat}.
In the further discussions, $E$ stands for a Banach space. Also
$K$ is the real or complex field and $I$ is the set of positive
integers. We first provide the required preliminaries, which
follow from \cite{kat}.

\begin{DEF}\em
For $1\leq p \leq \infty$, $1\leq q<\infty$ or $1\leq p<\infty$,
$q=\infty$ let $\ell_{p,q}\{E\}$ be the space of all $E$-valued
zero sequences $\{x_i\}$ such that
\begin{eqnarray*}
\|\{x_i\}\|_{p,q}= \left\lbrace
\begin{array}{c l}
\left(\sum_{i=1}^{\infty} i^{q/p -1} {\|x_{\phi(i)}\|}\;^q
\right)^{\frac{1}{q}} & \hbox{for $1\leq p\leq\infty, ~1\leq
q<\infty$}\\
\sup_i i^{\frac{1}{p}}\|x_{\phi(i)}\| & \hbox{for $1\leq
p<\infty,~ q=\infty$}.
\end{array}\right.
\end{eqnarray*}
is finite, where $\{\|x_{\phi(i)}\|\}$ is the non-increasing
rearrangement of $\{\|x_i\|\}$. If $E=K$, then $\ell_{p,q}\{K\}$
is denoted by $\ell_{p,q}$.
\end{DEF}
In particular, $\ell_{p,p}\{E\}$ coincides with $\ell_p\{E\}$ and
$\|.\|_{p,p}= \|.\|_p$; see \cite{mi}.\\

The following result will be used in the final result of this
paper. It is in fact \cite[Proposition 2]{kat}.

\begin{pro}\label{p22}
Let $E$ be a Banach space.
\begin{enumerate}
\item[(i)] If $1 \leq p < \infty , 1 \leq q<q_1 \leq \infty$, then
$\ell_{p,q}\{E\} \subseteq \ell_{p,q_1}\{E\}$ and for every
$\{x_i\} \in \ell_{p,q}\{E\}$
$$\|\{x_i\}\|_{p,q_1} \leq \left(\frac{q}{p}\right)^{\frac{1}{q}-\frac{1}{q_1}}
\|\{x_i\}\|_{p,q},$$ for $p<q$ and $$\|\{x_i\}\|_{p,q_1} \leq
\|\{x_i\}\|_{p,q},$$ for $p \geq q$. In fact
$$\|\{x_i\}\|_{p,q_1} \leq \max \{1,
\frac{q}{p}\}\|\{x_i\}\|_{p,q}.$$ \item[(ii)] Let eather $1\leq p
<p_1\leq\infty$, $1\leq q<\infty $ or $1 \leq p<p_1<\infty$,
$q=\infty$. Then
$$\ell_{p,q}\{E\} \subseteq\ell_{p_1,q}\{E\} $$
and for every $\{x_i\} \in \ell_{p,q}\{E\}$
$$\|\{x_i\} \|_{p_1,q}\leq\|\{x_i\}\|_{p,q}$$
\end{enumerate}
\end{pro}

Now for $J,Q\subseteq [1,\infty)$ let
\begin{center}
$IL_{J,Q}=\{\{x_i\}\in\bigcap_{p\in J,q\in Q}\ell_{p,q}:\;\;
\|\{x_i\}\|_{J,Q}=\sup_{p\in J,q\in Q}
\|\{x_i\}\|_{{p,q}}<\infty\}$.
\end{center}
We finish this work with the following result, which determines
the structure of $IL_{J,Q}$.
\begin{pro}\label{p15}
Let $J,Q\subseteq [1,\infty)$. Then $IL_{J,Q}=\ell_{m_J,m_Q}$.
\end{pro}

\begin{proof}
Some similar arguments to \cite[Theorem 12]{aalr2} implies that
$IL_{p,Q}=\ell_{p,m_Q}$. Indeed, by Proposition \ref{p22} for each
$q\in Q$, $\ell_{p,m_Q} \subseteq \ell_{p,q}$. Also for each
$\{x_i\}\in \ell_{p,m_Q}$,
$$\|\{x_i\}\|_{p,q} \leq \max\{ 1, \frac{m_Q}{p}\} \|\{x_i\}\|_{p,m_Q}.$$
It follows that $\{x_i\}\in IL_{p,Q}$ and
$$\|\{x_i\}\|_{p,Q} \leq \max \{1,\frac{m_Q}{p}\}\|\{x_i\}\|_{p,m_Q}.$$
Thus $\ell_{p,m_Q}\subseteq IL_{p,Q}$. The reverse of this
inclusion is clear whenever $m_Q\in Q$. Now let $m_Q\notin Q$.
Thus there is a sequence $(y_n)_{n\in\Bbb N}$ in $Q$, converging
to $m_Q$. For each $\{x_i\}\in IL_{p,Q}$, Fatou's lemma implies
that
\begin{eqnarray*}
\|\{x_i\}\|_{p,m_Q}^{m_Q} &=& \sum_{i=1}^{\infty} i^{\frac{m_Q}{p}-1} \|x_{\Phi(i)}\|^{m_Q}\\
&=& \sum_{i=1}^{\infty} \lim\inf_n \left( i^{\frac{y_n}{p}-1} \|x_{\Phi(i)}\|^{y_n} \right)\\
& \leq &\lim\inf_n \sum_{i=1}^{\infty} \left(i^{\frac{y_n}{p}-1} \|x_{\Phi(i)}\|^{y_n} \right)\\
&=& \lim\inf_n \|\{x_i\}\|_{p,y_n}^{y_n} \leq \lim\inf_n \|\{x_i\}\|_{p,Q}^{y_n}\\
&=& \|\{x_i\}\|_{p,Q}^{m_Q},
\end{eqnarray*}
which implies $\{x_i\}\in\ell_{p,m_Q}$. Consequently
$IL_{p,Q}=\ell_{p,m_Q}$. In the sequel, we show that
$IL_{J,q}\subseteq\ell_{m_J,q}$, for each $q\in Q$. Suppose that
$(y_n)$ be a sequences in $J$ such that $\lim_n x_n=m_J$ and
$\{x_i\}\in IL_{J,q}$ Then by Fatou's lemma, we have
\begin{eqnarray*}
\|\{x_i\}\|_{m_J,q}^q &=& \sum_{i=1}^{\infty} \left(i^{\frac{q}{m_J}-1} \|x_{\Phi(i)}\|^q\right)\\
&=&\sum_{i=1}^{\infty}\lim\inf_n \left(i^{\frac{q}{y_n}-1} \|x_{\Phi(i)}\|^q\right)\\
&\leq& \lim\inf_n \sum_{i=1}^{\infty}\left(i^{\frac{q}{y_n}-1}
\|x_{\Phi(i)}\|^q\right)\\
&=&\lim\inf_n \|\{x_i\}\|_{y_n,q}^q\\
&\leq&\|\{x_i\}\|_{J,q}^q\\
&<&\infty.
\end{eqnarray*}
Hence $IL_{J,q}\subseteq\ell_{m_J,q}$. Now suppose that $\{x_i\}
\in IL_{J,Q}$. Then for each $q\in Q$, $ \{x_i\}\in IL_{J,q}$ and
so $\{x_i\}\in\ell_{m_J,q}$. On the other hand by the above
inequalities, for each $1\leq q<\infty$, we have
$\|\{x_i\}\|_{m_J,q} \leq \|\{x_i\}\|_{J,q}$. So $\{x_i\}\in
IL_{m_J,Q} \subseteq \ell_{m_J,m_Q}$, which implies $IL_{J,Q}
\subseteq \ell_{m_J,m_Q}$. Also by Proposition \ref{p22}, for each
$p\geq m_J$ and $q\geq m_Q$ we have $\ell_{m_J,m_Q} \subseteq
\ell_{m_J,q} \subseteq\ell_{p,q}$. Consequently
$$\ell_{m_J,m_Q} \subseteq \bigcap_{p\in J , q\in Q} \ell_{p,q}.$$
Moreover for each $\{x_i\}\in \ell_{p,q}$,
$$
\sup_{p\in J, q\in Q} \|\{x_i\}\|_{p,q} \leq \sup_{q\in Q}
\|\{x_i\}\|_{m_J,q} \leq \max \{ 1, \frac{m_Q}{m_J}\}
\|\{x_i\}\|_{m_J,m_Q}.
$$
It follows that $$\ell_{m_J,m_Q}\subseteq IL_{J,Q} \subseteq
\ell_{m_J,m_Q}.
$$
Therefore $IL_{J,Q}=\ell_{m_J,m_Q}$, as claimed.
\end{proof}

{\bf Acknowledgment.} This research was partially supported by the
Banach algebra Center of Excellence for Mathematics, University of
Isfahan.

\vspace{9mm}

{\footnotesize \noindent
 F. Abtahi\\
  Department of Mathematics,
   University of Isfahan,
    Isfahan, Iran\\
     f.abtahi@sci.ui.ac.ir\\

\noindent
 H. G. Amini\\
  Department of Mathematics,
   University of Isfahan,
    Isfahan, Iran\\

\noindent
 A. Lotfi\\
  Department of Mathematics,
   University of Isfahan,
    Isfahan, Iran\\
     hali-lotfi@yahoo.com\\

\noindent
 A. Rejali\\
  Department of Mathematics,
   University of Isfahan,
    Isfahan, Iran\\
     rejali@sci.ui.ac.ir\\


\footnotesize

\begin{thebibliography}{99}

\bibitem{aalr1} Abtahi, F., Amini, H. G., Lotfi, H. A. and Rejali,
A., {\it An arbitrary intersection of $L_{p}-$spaces}, Bull. Aust.
Math. Soc., {\bf 85}, (2012), 433-445.

\bibitem{aalr2} Abtahi, F., Amini, H. G., Lotfi, H. A. and Rejali,
A., {\it Some intersections of the weighted $L_{p}-$spaces},
Abstr. Appl. Anal., Article ID 986857, 12 page, 2013.


\bibitem{gr} Grafakos, L., Classical Fourier Analysis,
2nd edn., Springer Science+Business Media, LLC, (2008).

\bibitem{kat} Kato, M.,
{\it On Lorentz spaces $ l_{p,q}(E)$}, Hiroshima Math. J., {\bf
6}, (1976), 73-93.

\bibitem{mi} Miyazaki, K.,
{\it $(p-q)$- nuclear and $(p-q)$- integral operators}, Hiroshima
Math. J., {\bf 4}, (1974), 99-132.

\bibitem{r} Rudin, W., Real and complex analysis, third edn, McGraw-Hill Book
Co., New York, 1987.

\end{thebibliography}
\end{document}